\documentclass[11pt]{amsart}
\usepackage{amssymb}
\usepackage[left=1in,top=1.3in,right=1in,bottom=1in,head=.3in,foot=.3in]{geometry}

\usepackage{comment}

\theoremstyle{plain}
\newtheorem{theorem}{Theorem}[section]
\newtheorem{lemma}[theorem]{Lemma}
\newtheorem{prop}[theorem]{Proposition}
\newtheorem{cor}[theorem]{Corollary}

\newtheorem{fact}[theorem]{Fact}

\theoremstyle{definition}
\newtheorem{defn}[theorem]{Definition}
\newtheorem{example}[theorem]{Example}

\newcommand{\K}{\mathcal K}
\newcommand{\m}{\mathfrak m}
\newcommand{\M}{\mathcal M}
\newcommand{\ntp}{NTP\textsubscript{2} }
\newcommand{\Ov}{\mathcal O_v}

\DeclareMathOperator{\ac}{ac}
\DeclareMathOperator{\bdn}{bdn}
\DeclareMathOperator{\id}{id}

\DeclareMathOperator{\supp}{supp}
\DeclareMathOperator{\tp}{tp}

\DeclareMathOperator{\Th}{Th}
\newcommand{\VF}{{\operatorname{VF}}}
\newcommand{\VG}{{\operatorname{VG}}}
\newcommand{\RF}{{\operatorname{RF}}}

\renewcommand{\L}{\mathcal L}
\newcommand{\Ldiv}{{\L_{\operatorname{div}}}}
\newcommand{\Lpas}{{\L_{\operatorname{Pas}}}}
\newcommand{\Lring}{{\L_{\operatorname{ring}}}}

\newcommand{\Lvg}{{\L_\VG}}
\newcommand{\Lrf}{{\L_\RF}}

\newcommand{\newcase}[1]{\bigskip\noindent\emph{Case #1:}}
\newcommand{\nogapcase}[1]{\noindent\emph{Case #1:}}

\begin{document}

\title{Burden of henselian valued fields in the {Denef-Pas} language}
\author{Peter Sinclair}
\address{Peter Sinclair, Department of Mathematics, Douglas College, 700 Royal Ave, New Westminster, British Columbia V3M 5Z5, Canada}
\email{sinclairp@douglascollege.ca}

\begin{abstract}
	Motivated by the Ax-Kochen/Ershov principle, a large number of questions about henselian valued fields have been shown to reduce to analogous questions about the value group and residue field. In this paper, we investigate the burden of henselian valued fields in the three-sorted Denef-Pas language. If $T$ is a theory of henselian valued fields admitting relative quantifier elimination (in any characteristic), we show that the burden of $T$ is equal to the sum of the burdens of its value group and residue field. As a consequence, $T$ is NTP\textsubscript{2} if and only if its residue field and value group are; the same is true for the statements ``$T$ is strong'' and ``$T$ has finite burden.''
\end{abstract}

\maketitle

\begin{center}
\begin{minipage}{.7\textwidth}
\setcounter{tocdepth}{1}
\tableofcontents
\end{minipage}
\end{center}


\section{Introduction}
\label{introduction}


The Ax-Kochen/Ershov principle states that the theory of a henselian valued field of equicharacteristic 0 is completely determined by the theory of its residue field and value group. It has since been extended to apply to certain henselian valued fields with positive residue characteristic (see \cite{AJ19}, \cite{Bel99}, \cite{Kuhl16}), and is a good heuristic even in fields where it does not hold: many useful properties of henselian valued fields are witnessed by some combination of the residue field and value group.

One area where this line of thinking has been particularly effective is in the classification theory of Shelah \cite{She90}, using combinatorial properties such as NIP (not the independence property, sometimes called dependence) and \ntp (not the tree property of the second kind). Delon \cite{Del78} showed that a henselian valued field of equicharacteristic zero is NIP if and only if its residue field is; this result was later extended to certain valued fields of positive residue characteristic in \cite{Bel99} and \cite{JS20}. In \cite{Cher14}, Chernikov showed an analogous result for \ntp in equicharacteristic zero.

Chernikov's result uses Adler's notion of burden, which measures the complexity of types using independent partition patterns (inp-patterns, described in Section \ref{inppatterns} below). Burden is equivalent to weight in simple theories and to dp-rank in NIP theories, and can be used to measure the complexity of a theory by considering the partial type $\{x=x\}$: a theory is \ntp if and only if there is a cardinal $\kappa$ such that every inp-pattern in $\{x=x\}$ has depth at most $\kappa$. Chernikov showed that in the three-sorted Denef-Pas language, the burden of a henselian valued field of equicharacteristic zero can be bounded by the burden of the induced structure on the residue field and value group, from which the \ntp transfer principle follows immediately.

The main result of this paper improves the precision of Chernikov's bound and extends the result to apply to theories of any characteristic, provided the theory eliminates quantifiers in the field sort.

{
	\def\thetheorem{\ref{mainthm}}
	\begin{theorem}
		Suppose $T$ is a theory of henselian valued fields in $\Lpas$ admitting relative quantifier elimination. Then
		\[ \bdn(T) = \bdn(T_{\VG}) + \bdn(T_{\RF}) ,\]
		where $T_{\VG}$ and $T_{\RF}$ are the induced theories on the value group and residue field, respectively.
	\end{theorem}
}

Touchard \cite{Tou18} found similar bounds for certain classes of henselian valued fields using the RV-structure, rather than the angular component map, in a recent preprint based on work of Chernikov and Simon \cite{CS19}. Specifically, the burden of a valued field $(K,v)$ is equal to the burden of its RV-sorts, which is in certain cases equal to the maximum of $\bdn(T_{\VG})$ and $\bdn(T_{\RF})$. If any of these burdens are infinite, this bound is identical to the bound found in this paper, although Touchard's bounds are tighter in the case when all of the burdens are finite and the additional assumptions can be made.

Section \ref{prelim} summarizes the necessary definitions and facts about valued fields and burden. Then in Section \ref{relativeqe}, we generalize a result of Delon \cite{Del78} classifying types in henselian valued fields of equicharacteristic 0, which we use to obtain an improved version of relative quantifier elimination. This result has previously been extended by B\'elair \cite{Bel99} to algebraically maximal Kaplansky valued fields of equicharacteristic $p$ and unramified henselian valued fields of mixed characteristic.

Finally, Section~\ref{calculating} uses the improved quantifier elimination result to prove the main result of the paper, Theorem \ref{mainthm}. This section proceeds via a series of propositions that allow us to restrict our focus to increasingly tame inp-patterns.

This paper is based on results from a chapter of the author's thesis \cite{Sin18}, under the supervision of Professor Deirdre Haskell.


\section{Preliminaries}
\label{prelim}


\subsection{Algebra of Valued Fields}
\label{aovf}


We assume that the reader is familiar with the basic notions of valued fields. For more detail, refer to any textbook on valued fields, such as \cite{EP05}. Given a valuation $v$ on a field $K$, we denote the value group by $vK$, the residue field by $Kv$, the valuation ring by $\Ov$, and the maximal ideal by $\m_v$. If $(K,v)$ and $(K',v')$ are valued fields then a valued field isomorphism from $(K,v)$ to $(K',v')$ is a field isomorphism $\phi: K\to K'$ satisfying $v(x) < v(y)$ if and only if $v(\phi(x)) < v(\phi(y))$ for all $x,y\in K$. 

A particularly useful class of valued fields is the class of henselian valued fields. A valuation $v$ on a field $K$ is called henselian if one of the following equivalent conditions holds:
\begin{enumerate}
	\item There is a unique valuation $w$ on the algebraic closure of $K$ such that $v = w|_K$
	\item Every polynomial $p(X) = X^n + aX^{n-1} + \sum_{i=0}^{n-2}a_iX^i \in \Ov[X]$ with $v(a) = 0$ and $v(a_i) > 0$ for all $i$ has a root in $K$.
\end{enumerate}

In this case, we call $(K,v)$ a henselian valued field. It is clear from property (2) above that any valuation on a separably closed field is henselian; many other valued fields are also henselian, including the $p$-adic numbers with the usual valuation. It can be shown that every valued field has a minimal algebraic extension that is henselian, and that this extension is unique up to valued field isomorphism. This extension, denoted $(K^h,v^h)$, is called the henselization of $(K,v)$, and is always an immediate extension of $(K,v)$.

\begin{example}
	(Field of Hahn series) Let $k$ be any field and $\Gamma$ be any group, and consider the set $k[[t^\Gamma]]$ of functions $f: \Gamma\to k$ such that $\supp(f) = \{\gamma\in\Gamma : f(\gamma)\neq 0\}$ is well-ordered. We think of elements of this set as power series and write them as $f = \sum_{\gamma\in\Gamma} a_\gamma t^\gamma$, where $a_\gamma = f(\gamma)$. This set is a field with the usual operations on power series, and the map $v(f) = \min(\supp(f))$ is a henselian valuation on $k[[t^\Gamma]]$.
\end{example}

\begin{defn}
	Let $(a_\rho)_{\rho< I}$ be a sequence of elements of a valued field $(K,v)$ indexed by any ordered set $I$. We say that $(a_\rho)$ is \emph{pseudo-convergent} if
	\[ v(a_{\rho_2}-a_{\rho_1}) < v(a_{\rho_3}-a_{\rho_2}) \]
	for all $\rho_1<\rho_2<\rho_3$. If $(a_\rho)$ is pseudo-convergent then for each $\rho\in I$ there exists $\gamma_\rho\in vK$ such that
	\[ v(a_\rho - a_{\rho'}) = \gamma_\rho \]
	for all $\rho' > \rho$. We say that $x\in K$ is a \emph{pseudo-limit} of $(a_\rho)$ if $v(x-a_\rho) = \gamma_\rho$ for all $\rho$.
\end{defn}

Note that pseudo-limits are not unique. However, field extensions of the form $K(x)$, where $x$ is a pseudo-limit of a pseudo-convergent sequence in $K$, are unique up to valued field isomorphism in certain situations, such as when the sequence is of transcendental type.

\begin{defn}
	Let $I$ be a well-ordered set without a maximum element, and let $(a_\rho)_{\rho\in I}$ be a pseudo-convergent sequence in a valued field $(K,v)$. We say that $(a_\rho)$ is of \emph{transcendental type} if, for all $p(X)\in K[X]$,
	\[ v(p(a_{\rho_1})) = v(p(a_{\rho_2})) \]
	for all sufficiently large $\rho_1,\rho_2\in I$.
\end{defn}

More details can be found in \cite{Kap42}, the paper in which pseudo-convergent sequences were originally introduced.

\begin{defn}
	An \emph{angular component map} is a function $\ac: K\to Kv$ which satisfies the following:
	\begin{enumerate}
		\item $\ac(0) = 0$
		\item For all $x\in \Ov^\times$, $\ac(x) = x+\m$
		\item For all $x,y\in K$, $\ac(xy) = \ac(x)\ac(y)$.
	\end{enumerate}
\end{defn}

On a Hahn field, the map that returns the nonzero coefficient with minimum index of a power series is an angular component map. Not every valued field admits an angular component map, but every valued field has an elementary extension that does \cite[Corollary 1.6]{Pas90}.

The two facts below summarize some fundamental results about the relationship between valuations and angular component maps; they follow easily from the definitions and will be used repeatedly in Section~\ref{calculating}.

\begin{fact} \label{orthogonal}
	Suppose $(K,v)$ is a valued field. For every $\gamma\in vK$ and $r\in Kv^\times$, there exists $a\in K$ with $v(a) = \gamma$ and $\ac(a) = r$.
\end{fact}

\begin{fact} \label{12-easy}
	Suppose $(K,v)$ is a valued field and $a,b,c\in K^\times$.
	\begin{enumerate}
		\item If $v(a-b) < v(c-b)$ then $v(a-b) = v(a-c)$ and $\ac(a-b) = \ac(a-c)$.
		\item If $v(a-b) = v(a-c)$ then $\ac(a-b)\neq \ac(a-c)$ if and only if $v(a-b) = v(c-b)$.
		\item If $v(a-b) = v(a-c) = v(c-b)$ then $\ac(a-c) = \ac(a-b) - \ac(c-b)$.
	\end{enumerate}
\end{fact}


\subsection{Model Theory of Valued Fields}
\label{mtovf}


Valued fields can be viewed as first order structures in a number of ways; in this paper, we exclusively view a valued field as a three-sorted structure with sorts for $K$, $vK$, and $Kv$, and various maps between them.

\begin{defn}
	The \emph{Denef-Pas Language} for valued fields is the three-sorted language $\Lpas$ with the following sorts and functions:
	\begin{enumerate}
		\item The valued field sort $\VF$ has the language of rings $\Lring = \{0,1,+,-,\cdot\}$
		\item The value group sort $\VG$ has an expansion of the language of ordered abelian groups $\Lvg =     \{0,+,-,<,\infty,\ldots\}$
		\item The residue field sort $\RF$ has an expansion $\Lrf$ of the language of rings
		\item The only maps between sorts are $v: \VF\to \VG$ and $\ac: \VF\to\RF$.
	\end{enumerate}
\end{defn}

Calling this ``the'' Denef-Pas Language is slightly misleading, since the value group and residue field languages are some expansion of the appropriate minimum languages. When we consider a valued field $(K,v)$ as an $\Lpas$-structure, we always assume that the $\VF$-sort is $K$, the $\VG$-sort is $vK$, the $\RF$-sort is $Kv$, $v$ is the valuation map, and $\ac$ is an angular component map.

We say that a theory $T$ in $\Lpas$ admits \emph{relative quantifier elimination} if it eliminates quantifiers $\forall x$ and $\exists x$, where $x$ is a variable in the valued field sort. It follows syntactically from relative quantifier elimination that every formula $\phi(x^\VF, x^\VG, x^\RF)$ in $T$ is equivalent to one of the form
\[ \bigvee_{i=1}^n \chi_i(v(f_1(x^\VF),\ldots, v(f_m(x^\VF)), x^\VG) \wedge \rho_i(\ac(f_1(x^\VF)),\ldots,\ac(f_m(x^\VF)), x^\RF) \]
where $x^\VF,x^\VG,x^\RF$ are tuples of variables in the sorts $\VF,\VG,\RF$, respectively, $\chi_i$ are $\Lvg$-formulas, $\rho_i$ are $\Lrf$ formulas, and $f_j$ are polynomials with integer coefficients. Note that there is no $\Lring$-formula corresponding to the $\VF$-sort; this is because any such formula would be a boolean combination of statements of the form $g(x^\VF) = 0$, which is equivalent to $v(g(x^\VF)) = \infty$, and so this part of the formula can be absorbed into the $\Lvg$ portion.

Suppose $T$ is an $\Lpas$-theory with relative quantifier elimination, and consider the special case of a formula $\phi(x)$ with parameters in some model $(K,v)$ such that $x$ is a singleton in the $\VF$-sort. In this case, $\phi(x)$ is equivalent to a formula of the form
\[ \bigvee_{i=1}^n \chi_i(v(f_1(x),\ldots, v(f_m(x))) \wedge \rho_i(\ac(f_1(x)),\ldots,\ac(f_m(x))) \]
where $\chi_i$ are $\Lvg$-formulas with parameters in $vK$, $\rho_i$ are $\Lrf$ formulas with parameters in $Kv$, and $f_j$ are polynomials with coefficients in $K$. This follows immediately from the general form of relative quantifier elimination by substituting a parameter for every variable except a singleton in the $\VF$-sort.

Many theories of henselian valued fields have relative quantifier elimination, including all theories of henselian valued fields of characteristic $(0,0)$ \cite{Pas89}, algebraically maximal Kaplansky fields of characteristic $(p,p)$ \cite{Bel99}, and strongly dependent henselian valued fields in any characteristic \cite{HH19}. In the case where $T$ is a theory of henselian valued fields of characteristic $(0,0)$, we may assume that the polynomials $f_i$ are all linear by the cell decomposition of \cite{Pas89}. In fact, we prove in Section \ref{relativeqe} that this is true in any characteristic.

The celebrated Ax-Kochen/Ershov (AKE) principle can be viewed as an immediate consequence of relative quantifier elimination.

\begin{fact} \label{AKE}
	Suppose $(K,v)$ and $(L,w)$ are both models of some theory $T$ of henselian valued fields in $\Lpas$ that admits relative quantifier elimination. Then $(K,v) \equiv (L,w)$ if and only if $vK \equiv wL$ (as $\Lvg$-structures) and $Kv \equiv Lw$ (as $\Lrf$-structures).
\end{fact}


\subsection{Inp-patterns and Burden}
\label{inppatterns}


Burden is a notion of complexity of a partial type developed originally by Adler \cite{Adl07} as a generalization of the notion of weight from simple theories. Conveniently, it also generalizes the notion of dp-rank from NIP theories.

\begin{defn}
	Let $\pi(x)$ be a partial type and $\kappa$ a cardinal. An \emph{inp (independent partition) pattern} in $\pi(x)$ of depth $\kappa$ consists of tuples $\{b_{\alpha,i} : \alpha<\kappa, i<\omega\}$, formulas $\{\phi_\alpha(x,y_\alpha): \alpha<\kappa\}$ with $|y_\alpha| = |b_{\alpha,i}|$, and $k_\alpha < \omega$ such that
	\begin{enumerate}
		\item $\{\phi_\alpha(x,b_{\alpha,i})\}_{i<\omega}$ is $k_\alpha$-inconsistent for each $\alpha<\kappa$
		\item $\pi(x)\cup \{\phi_\alpha(x,b_{\alpha, \eta(\alpha)})\}_{\alpha<\kappa}$ is consistent for any $\eta : \kappa\to\omega$.
	\end{enumerate}
	The \emph{burden} of $\pi(x)$, written $\bdn(\pi)$, is the supremum of the depths of all inp-patterns in $\pi(x)$, or $\infty$ if the supremum does not exist.
\end{defn}

In order to simplify the notation, we often write $(\phi_\alpha(x,y_\alpha), b_\alpha, k_\alpha)_{\alpha<\kappa}$ for the above inp-pattern. In this notation, $b_\alpha$ represents the sequence $(b_{\alpha,i})_{i<\omega}$. By strengthening the assumptions on inp-patterns slightly, we can make it easier to check whether a given array is an inp-pattern.

\begin{defn}
	Let $\pi(x)$ be a partial type and $\kappa$ a cardinal. An \emph{indiscernible inp-pattern} in $\pi(x)$ of depth $\kappa$ consists of tuples $\{b_{\alpha,i} : \alpha<\kappa, i<\omega\}$ and formulas $\{\phi_\alpha(x,y_\alpha): \alpha<\kappa\}$ with $|y_\alpha| = |b_{\alpha,i}|$ such that
	\begin{enumerate}
		\item The sequences $(b_{\alpha,i})_{i<\omega}$ are mutually indiscernible; that is, for each $\alpha<\kappa$, the sequence $(b_{\alpha,i})_{i<\omega}$ is indiscernible over $\{b_{\beta,i} : \beta\neq\alpha, i<\omega\}$
		\item $\{\phi_\alpha(x,b_{\alpha,i})\}_{i<\omega}$ is inconsistent for each $\alpha<\kappa$
		\item $\pi(x)\cup \{\phi_\alpha(x,b_{\alpha, 0})\}_{\alpha<\kappa}$ is consistent.
	\end{enumerate}
\end{defn}

As with inp-patterns, we often condense the notation for the above indiscernible inp-pattern to $(\phi_\alpha(x,y_\alpha), b_\alpha)_{\alpha<\kappa}$.

It follows immediately from the definition of indiscernibility that every indiscernible inp-pattern is an inp-pattern. By a common argument using Ramsey theory and compactness, any inp-pattern can be used to generate an indiscernible inp-pattern of the same depth; see Lemma 5.1.3 of \cite{TZ12} for a more detailed explanation. Thus, the burden of a partial type $\pi(x)$ is equal to the supremum of the depths of all indiscernible inp-patterns in $\pi(x)$.

The burden of $\pi(x)$ measures the complexity of $\pi$: the greater the depth $\kappa$ of an inp-pattern, the closer $\pi$ is to satisfying TP\textsubscript{2}, the tree property of the second kind. In fact, $\bdn(\pi) = \infty$ if and only if $\pi(x)$ has TP\textsubscript{2}.

We can measure the complexity of a theory by looking at the partial type $\pi(x)$ that contains only the formula $x=x$, where $x$ is a singleton; when we write $\bdn(T)$, we mean $\bdn(\pi)$ for this choice of $\pi(x)$. If $T$ is a theory in a multi-sorted language (say $\Lpas$), then there is a separate formula $x=x$ for each sort and we take $\bdn(T)$ to be the supremum over the sorts. In the case of $\Lpas$, we can restrict our focus to the valued field sort $\VF$ because there are definable surjections from $\VF$ to the other sorts; this is described in more detail in Section \ref{calculating}.

Note that if $\bdn(T) = \aleph_0$ then either $\pi(x) = \{x=x\}$ has an inp-pattern of depth $\aleph_0$ or $T$ has inp-patterns of all finite (but no infinite) depths; in the latter case, we say that $T$ is \emph{strong}. This odd situation actually occurs whenever $\bdn(T)$ is an infinite cardinal; one way to handle it is with the following definition, modified from \cite{Adl07}.

\begin{defn}
	Let Card* be the class containing the cardinals and, for each limit cardinal $\kappa$, a new symbol $\kappa_-$. The ordering on cardinals is extended to Card* by setting $\kappa_-$ to be a predecessor to $\kappa$; that is, for all $\lambda\in\text{Card*}$, $\lambda < \kappa$ if and only if $\lambda \leq \kappa_-$.
\end{defn}

We can then modify the definition of burden so that $\bdn(T)$ is the supremum in Card*, or $\infty$ if the supremum does not exist; that way, a theory $T$ has $\bdn(T) = \aleph_{0-}$ if it is strong, and $\bdn(T) = \aleph_0$ if $\pi(x) = \{x=x\}$ has an inp-pattern of depth $\aleph_0$. The arguments given in this paper work whether burden is defined to be a cardinal or an element of Card*, provided the sum of two infinite elements in Card* is defined to be the maximum of those elements, just as it is for cardinals.


\section{Relative Quantifier Elimination}
\label{relativeqe}


In Section \ref{calculating}, we generalize and improve a result of \cite{Cher14} relating the burden of certain valued fields to the burdens of their value groups and residue fields. In order to obtain the generalization, we need a stronger version of relative quantifier elimination than the one given in Section \ref{mtovf}.

We begin with a classification of 1-types over any model in $\Lpas$, due to \cite{Del78}. Consider an elementary extension $\K\prec\M$ of valued fields in $\Lpas$, fix $x\in M\smallsetminus K$, and define
\[ I_K(x) = \{\gamma\in vK : \gamma = v(x-k) \text{ for some }k\in K\} .\]
Then $\tp(x/K)$ belongs to one of three families, depending on the structure of $I_K(x)$.
\begin{enumerate}
	\item $I_K(x) = \{v(x-k) : k\in K\}$ and does not have a maximum element. In this case, we say that $\tp(x/K)$ is \emph{immediate}.
	\item $I_K(x) = \{v(x-k) : k\in K\}$ and has a maximum element. In this case, we say that $\tp(x/K)$ is \emph{residual}.
	\item $I_K(x) \neq \{v(x-k) : k\in K\}$. In this case, we say that $\tp(x/K)$ is \emph{valuational}.
\end{enumerate}
In the first two cases, $\{v(x-k) : k\in K\}$ is a subset of $vK$. In the third, there is a single element $\gamma_0\in \{v(x-k) : k\in K\}\smallsetminus vK$; if there were two, say $\gamma_0 = v(x-k_0) < v(x-k_1) = \gamma_1$, then $v(x-k_0) = v((x-k_0)-(x-k_1)) = v(k_1-k_0) \in vK$, contradicting $\gamma_0\notin vK$. By a similar argument, $\gamma_0$ is an upper bound for $I_K(x)$.


The stronger form of relative quantifier elimination we need is a consequence of the following theorem.

\begin{theorem} \label{isolating}
	Suppose $\K$ is a henselian valued field in $\Lpas$ such that $\Th(\K)$ admits relative quantifier elimination. Let $\M$ be a monster model of $\Th(\K)$ and let $x\in M\smallsetminus K$ be an element of the valued field sort.
	\begin{enumerate}
		\item If $\tp(x/K)$ is immediate, let $(a_\rho,\gamma_\rho)_{\rho<\kappa}$ be a sequence such that $a_\rho\in K$, $\gamma_\rho = v(x-a_\rho)$, and $(\gamma_\rho)$ is strictly increasing and cofinal in $I_K(x)$. Then $\tp(x/K)$ is completely determined by the set of formulas $\{v(x-a_\rho)=\gamma_\rho : \rho<\kappa\}$.
		\item If $\tp(x/K)$ is residual, then it is completely determined by a pair of constants $a\in K$ and $\gamma\in vK$ such that $v(x-a) = \gamma$ and $\ac(x-a)\notin Kv$, by the formula $v(x-a) = \gamma$, and by the type $\tp(\ac(x-a)/Kv)$.
		\item If $\tp(x/K)$ is valuational, then it is completely determined by some constant $a\in K$ such that $v(x-a)\notin vK$, by the type $\tp(v(x-a)/vK)$, and by the type $\tp(\ac(x-a)/Kv)$.
	\end{enumerate}
\end{theorem}

In equicharacteristic 0, this theorem was originally proved by Delon \cite{Del78}; a more detailed proof using angular component maps can be found in \cite{BB96}. B\'elair later extended the result to certain fields of characteristic $(p,p)$ and $(0,p)$ \cite{Bel99}. Before we prove the result for any characteristic, we state the following technical lemma.

\begin{lemma} \label{automorph}
	Suppose $\K$ is a henselian valued field in $\Lpas$ such that $\Th(\K)$ admits relative quantifier elimination. Let $\M$ be a monster model of $\Th(\K)$ and suppose there are $y,y'\in M$ such that the following exist:
	\begin{enumerate}
		\item A valued field isomorphism $\phi: K(y)\to K(y')$ with $\phi|_K = \id_K$ and $\phi(y) = y'$
		\item An $\Lvg$-automorphism $\alpha: vM\to vM$ with $\alpha|_{vK} = \id_{vK}$ and $\alpha(v(y)) = v(y')$
		\item An $\Lrf$-automorphism $\beta: Mv\to Mv$ with $\beta|_{Kv} = \id_{Kv}$ and $\beta(\ac(y)) = \ac(y')$
	\end{enumerate}
	Assume moreover that the value group $v(K(y))$ is generated by $vK\cup\{v(y)\}$, that $v(K(y'))$ is generated by $vK\cup\{v(y')\}$, and that either
	\begin{enumerate}
		\item $\ac(y)$ and $\ac(y')$ are both transcendental over $Kv$, or
		\item $v(y^n)\notin vK$ and $v((y')^n)\notin vK$ for any nonzero $n\in\mathbb Z$.
	\end{enumerate}
	Then there exists an $\Lpas$-automorphism $\sigma$ of $\M$ with $\sigma|_{K(y)} = \phi$; in particular, this means $\tp(y/K) = \tp(y'/K)$.
\end{lemma}

\begin{proof}
	First, note that since $\phi$ is a valued field automorphism, by choice of $\alpha$ we have $\alpha(v(x)) = v(\phi(x))$ for all $x\in K(y)$. We claim that we also have $\beta(\ac(x)) = \ac(\phi(x))$ for all $x\in K(y)$. To prove this, we will first show that for every polynomial $p(X)\in K[X]$, there exists a polynomial $\bar p(X)\in Kv[X]$ such that $\ac(p(y)) = \bar p(\ac(y))$. Note that $\bar p(X)$ will not in general be the residue polynomial of $p(X)$, but a separate polynomial as described below.
	
	Suppose $v(y^n) \notin vK$ for any nonzero $n\in\mathbb Z$ and fix a polynomial $p(X)\in K[X]$. If two distinct terms of $p(y)$, say $z_1y^{n_1}$ and $z_2y^{n_2}$ have the same valuation, then we must have $v(y^{n_2-n_1}) = v(z_1/z_2) \in vK$, which is impossible. Thus, $p(y)$ has a term $zy^n$ of least valuation and $\ac(p(y)) = \ac(z)\ac(y)^n$, a polynomial in $Kv[\ac(y)]$.
	
	On the other hand, suppose $\ac(y)$ is transcendental over $Kv$. In this case, we proceed by induction on the degree of $p$. If $\deg(p) = 0$ then $p(X) = z$ for some $z\in K$ and $\ac(p(y)) = \ac(z)$. If $\deg(p) = n > 0$, then we can write $p(X) = z + Xq(X)$ for some $z\in K$ and some polynomial $q(X)$ of degree less than $n$. By induction, $\ac(q(y)) = \bar q(\ac(y))$ for some polynomial $\bar q$. Since $\ac(y)$ is transcendental over $Kv$, we must have $\ac(z) \neq \ac(yq(y)) = \ac(y)\bar q(\ac(y))$, and so $\ac(p(y))$ must be one of $\ac(z)$, $\ac(y)\bar q(\ac(y))$, or $\ac(z)+\ac(y)\bar q(\ac (y))$, depending on the relationship between $v(z)$ and $v(yq(y))$. In any case, $\ac(p(y))$ is a polynomial in $Kv[\ac(y)]$, completing the induction.
	
	In both cases, we showed that $\ac(p(y)) = \bar p(\ac(y))$ for some polynomial $\bar p(X)\in Kv[X]$, as desired. Note that the process of determining $\bar p(X)$ depended only on the original polynomial $p(X)$ and the valuations of the terms of $p(y)$. Because $\alpha(v(x)) = v(\phi(x))$ for all $x\in K(y)$, an identical argument shows that the same polynomial $\bar p(X)$ satisfies $\ac(p(y')) = \bar p(\ac(y'))$. Then, given any polynomial $p(X)\in K[X]$, we have
	\[ \beta(\ac(p(y))) = \beta(\bar p(\ac(y))) = \bar p(\beta(\ac(y))) = \bar p(\ac(y')) = \ac(p(y')) = \ac(\phi(p(y))) \]
	by choice of $\beta$. Since every element $x\in K(y)$ can be written as a rational function in $y$ and the angular component map is multiplicative, we can easily extend this result to the entire field.
	
	Finally, by the above observations, relative quantifier elimination, and the fact that $\alpha$ and $\beta$ are elementary maps, it follows that $\phi: K(y) \to M$ is a partial elementary map, and hence can be extended to an automorphism $\sigma$ of $\M$.
\end{proof}

We can now prove Theorem \ref{isolating}. The proof follows the outline of \cite{Del78}, \cite{BB96}, and \cite{Bel99}, but with any references to the specific characteristic of the field replaced by Lemma \ref{automorph}.

\begin{proof}
	(of Theorem \ref{isolating})
	
	\newcase{1} Suppose $\tp(x/K)$ is immediate. Fix any strictly increasing well-ordered cofinal sequence $(\gamma_\rho)_{\rho<\kappa}$ of $I_K(x)$ and any sequence $(a_\rho)_{\rho<\kappa}$ with $a_\rho\in K$ such that $v(x-a_\rho) = \gamma_\rho$. We claim that the the set of formulas $\{v(x-a_\rho) = \gamma_\rho : \rho<\kappa\}$ completely determines $\tp(x/K)$.
	
	Note that by choice of $a_\rho$ and $\gamma_\rho$, for $\rho_1<\rho_2<\kappa$, we have
	\[ v(a_{\rho_2}-a_{\rho_1}) = v((a_{\rho_2}-x)+(x-a_{\rho_1})) = \min\{v(a_{\rho_2}-x),\ v(x-a_{\rho_1})\} = \gamma_{\rho_1} \]
	since $\gamma_{\rho_1} < \gamma_{\rho_2}$. Thus, for $\rho_1<\rho_2<\rho_3<\kappa$, we have
	\[ v(a_{\rho_2}-a_{\rho_1}) = \gamma_{\rho_1} < \gamma_{\rho_2} = v(a_{\rho_3}-a_{\rho_2}) \]
	and so $(a_\rho)_{\rho<\kappa}$ is a pseudo-convergent sequence. Moreover, since $v(x-a_\rho) = \gamma_\rho$ for all $\rho<\kappa$, $x$ is a pseudo-limit of $(a_\rho)_{\rho<\kappa}$.
	
	Suppose $x'\in M$ is another element with $v(x'-a_\rho) = \gamma_\rho$ for all $\rho<\kappa$. Then $x'$ is also a pseudo-limit of $(a_\rho)_{\rho<\kappa}$, and since $\K\prec\M$, both $x$ and $x'$ are pseudo-limits of transcendental type. Then by Theorem 2 of \cite{Kap42}, $K(x)$ and $K(x')$ are immediate extensions of $K$ and there exists a valued field isomorphism $\phi: K(x)\to K(x')$ fixing $K$ and sending $x$ to $x'$.
	
	Because $K(x)$ is an immediate extension of $K$, for any $y\in K(x)$, there must exist $b\in K$ with $\ac(b) = \ac(y)$ and $v(b)=v(y)$ by Fact~\ref{orthogonal}. Then
	\[ v(\phi(y)) = v(\phi(b)) = v(b) = v(y) . \]
	Moreover, since $\ac(b)=\ac(y)$, we must have $v(b-y) > v(b)$. Then
	\[ v(b-\phi(y)) = v(\phi(b-y)) > v(\phi(b)) = v(b) \]
	and so $\ac(\phi(y)) = \ac(b) = \ac(y)$.
	
	Thus, the value group map induced by $\phi$ is the identity on $vK = v(K(x)) = v(K(x'))$ and the residue field map induced by $\phi$ is the identity map on $Kv = (K(x))v = (K(x'))v$. It then follows from relative quantifier elimination that $\phi$ is a partial elementary map and can be extended to an automorphism $\sigma$ of $\M$. Because $\sigma(x) = \phi(x) = x'$, this automorphism demonstrates that $\tp(x/K) = \tp(x'/K)$ as desired.
	
	\newcase{2} Suppose $\tp(x/K)$ is residual; we must first show that there exists $a\in K$ and $\gamma\in vK$ as described in the theorem. Let $\gamma\in vK$ be the largest element of $I_K(x)$, and fix $a\in K$ such that $v(x-a) = \gamma$. If $\ac(x-a)\in Kv$ then there must exist some $b\in K$ with $\ac(x-a) = \ac(b)$ and $v(b) = \gamma$ by Fact \ref{orthogonal}. But then
	\[ v(x-(a+b)) = v((x-a)-b) > v(x-a) = \gamma , \]
	contradicting the maximality of $\gamma$. Thus, $\ac(x-a)\notin Kv$.
	
	Now, suppose $x'\in M$ is another element with $v(x'-a) = \gamma$, $\ac(x'-a) \notin Kv$, and $\tp(\ac(x-a)/Kv) = \tp(\ac(x'-a)/Kv)$. We wish to show that $\tp(x/K) = \tp(x'/K)$, which we will do by finding an $\Lpas$-automorphism of $\M$ that fixes $\K$ and maps $y=x-a$ to $y'=x'-a$.
	
	
	Since $\K\prec \M$ is an elementary extension, $\operatorname{acl}_\M(\K) = \K$; in particular, $K^{\operatorname{acl}}\cap M = K$, and so $y$ and $y'$ must both be transcendental over $K$. Thus, there is a field isomorphism $\phi: K(y)\to K(y')$ fixing $K$ and sending $y$ to $y'$. Moreover, $\phi$ is a valued field isomorphism since $v(y) = \gamma = v(y')$. Setting $\alpha: vM\to vM$ to be the identity automorphism, we have $\alpha(v(y)) = v(y')$.
	
	Since $\tp(\ac(x-a)/K) = \tp(\ac(x'-a)/K)$, there is an $\Lrf$-automorphism $\beta: Mv\to Mv$ with $\beta|_{Kv} = \id_{Kv}$ and $\beta(\ac(y)) = \ac(y')$. Finally, $\ac(y)$ and $\ac(y')$ must be transcendental over $Kv$ since $Kv\prec Mv$. Then by Lemma \ref{automorph}, $\tp(y/K) = \tp(y'/K)$, which implies $\tp(x/K) = \tp(x'/K)$.
	
	\newcase{3} Suppose $\tp(x/K)$ is valuational, fix any $a\in K$ with $v(x-a)\notin vK$, and suppose $x'\in M$ is another element with $v(x'-a)\notin vK$, $\tp(v(x-a)/vK) = \tp(v(x'-a)/vK)$, and $\tp(\ac(x-a)/Kv) = \tp(\ac(x'-a)/Kv)$. As in Case 2, it suffices to show that $\tp(y/K) = \tp(y'/K)$ for $y = x-a$ and $y' = x'-a$.
	
	Again following Case 2, $\K\prec \M$, which means $y$ and $y'$ are both transcendental over $K$ and there exists a field isomorphism $\phi: K(y)\to K(y')$ fixing $K$ and sending $y$ to $y'$. Moreover, we have $v(K(y)) = vK\oplus \mathbb Z v(y)$ and $v(K(y')) = vK\oplus \mathbb Z v(y')$ since $vK\prec vM$ and $v(y),v(y')\notin vK$; in particular, $v(K(y))$ is generated by $vK\cup\{y\}$ and $v(y^n) = nv(y) \notin vK$ for any $n\in\mathbb Z$, and similarly for $y'$. It then follows from Corollary 2.2.3 of \cite{EP05} that $\phi$ is a valued field isomorphism.
	
	Finally, by choice of $x'$, there exists an $\Lvg$-automorphism $\alpha$ of $vM$ that fixes $vK$ and such that $\alpha(v(y)) = v(y')$. Similarly, there exists an $\Lrf$-automorphism $\beta$ of $Mv$ that fixes $Kv$ and such that $\beta(\ac(y)) = \ac(y')$. Thus, by Lemma \ref{automorph}, $\tp(y/K) = \tp(y'/K)$, so $\tp(x/K) = \tp(x'/K)$.
\end{proof}

As a consequence of the above theorem, we can improve the equivalence of formulas provided by relative quantifier elimination.

\begin{prop} \label{betterqe}
	Suppose $\K$ is a henselian valued field in $\Lpas$ such that $\Th(\K)$ admits relative quantifier elimination. Let $\phi(x)$ be a formula in one valued field sort variable with parameters in $K$. Then $\phi(x)$ is equivalent to a finite disjunction of formulas of the form
	\[ \chi\left(v(x-c^1),\ldots,v(x-c^n),b^\VG\right) \wedge \rho\left(\ac(x-c^1),\ldots,\ac(x-c^n),b^\RF\right) \]
	where $\chi(x,\bar y)$ is an $\Lvg$-formula, $\rho(x,\bar y)$ is an $\Lrf$-formula, $c^1,\ldots,c^n$ are singletons in the $\VF$-sort, $b^\VF$ is a $\K$-tuple in the $\VG$-sort, and $b^\RF$ is a $\K$-tuple in the $\RF$-sort.
\end{prop}

\begin{proof}
	For the duration of this proof, we will refer to formulas of the form $\chi\wedge\rho$ as in the statement of the Proposition as \emph{good formulas}. Note that all of the formulas occurring in the conclusion of Theorem \ref{isolating} are good formulas:
	\begin{enumerate}
		\item If $x$ is immediate then each formula has the form $v(x-a_\rho) = \gamma_\rho$ with $a_\rho\in K$ and $\gamma_\rho\in vK$.
		\item If $x$ is residual then each formula is an element of $\tp(\ac(x-a)/Kv)$ with $a\in K$, and hence has the form $\rho\left(\ac(x-a), b^\RF\right)$.
		\item If $x$ is valuational then each formula is an element of $\tp(v(x-a)/vK)$ or $\tp(\ac(x-a)/Kv)$ with $a\in K$, and so is a good formula in either case.
	\end{enumerate}
	Moreover, by a simple rearrangement, the conjunction of a finite set of good formulas is itself a good formula.
	
	Let $\{p_\alpha : \alpha<\kappa\}$ be the set of complete $K$-types containing $\phi(x)$. By Theorem \ref{isolating}, for each $\alpha<\kappa$ there is a partial type $\pi_\alpha(x)$ consisting only of good formulas such that $\pi_\alpha \vdash p_\alpha$; in particular, $\pi_\alpha\vdash \phi$. By compactness, this implication only requires a finite subset of $\pi_\alpha(x)$; let $\psi_\alpha(x)$ be the conjunction of this finite set, and note that $\psi_\alpha(x)$ is a good formula by the observation above.
	
	Ranging over $\alpha$, we have $\phi \vdash \bigvee_{\alpha<\kappa} \psi_\alpha(x)$. Of course, this statement is not first-order, due to the infinite disjunction. However, by a standard compactness argument, we can find a finite set $\{\psi_{\alpha_i}, \ldots, \psi_{\alpha_n}\}$ such that
	\[ \K\models \phi(x) \leftrightarrow \bigvee_{i=1}^n \psi_{\alpha_i}(x) .\]
	Since each $\psi_\alpha(x)$ is a good formula, this shows that $\phi(x)$ is equivalent to a finite disjunction of good formulas, as desired.
\end{proof}


\section{Calculating Burden}
\label{calculating}


Throughout this section, we assume that $\K = (K,vK,Kv)$ is a sufficiently saturated model of some theory $T$ of henselian valued fields in $\Lpas$ that admits relative quantifier elimination. For example $T$ might be strongly dependent or a theory of fields of characteristic $(0,0)$. In \cite{Cher14}, Chernikov gives a bound for $\bdn(T)$ in terms of $\bdn(T_\VG)$ and $\bdn(T_\RF)$ in the characteristic $(0,0)$ case, but the proof in that paper uses a Ramsey theory argument, and so the bound is very imprecise. The goal of this section is to improve Chernikov's bound and extend the result to apply to theories of any characteristic. First, we repeat two results from that paper that we will use throughout this section.

\begin{fact} \label{boolean}
	\cite[Lemma 7.1]{Cher14}
	\begin{enumerate}
		\item If $(\phi_{\alpha,0}(x,y_{\alpha,0})\vee \phi_{\alpha,1}(x,y_{\alpha,1}),\ a_\alpha,\ k_\alpha)_{\alpha<\kappa}$ is an (indiscernible) inp-pattern, then
		\[ \big(\phi_{\alpha, f(\alpha)}(x, y_{\alpha,f(\alpha)}), a_\alpha, k_\alpha\big)_{\alpha<\kappa} \]
		is also an (indiscernible) inp-pattern for some $f: \kappa\to\{0,1\}$.
		\item Let $(\phi_\alpha(x, y_\alpha), a_\alpha, k_\alpha)_{\alpha<\kappa}$ be an (indiscernible) inp-pattern and assume that
		\[ \phi_\alpha(x, a_{\alpha,i}) \leftrightarrow \psi_\alpha(x, b_{\alpha,i}) \]
		for all $\alpha<\kappa$, all $i<\omega$, and some (mutually indiscernible) $( b_\alpha)_{\alpha<\kappa}$. Then there is an (indiscernible) inp-pattern of the form $(\psi_\alpha(x, z_\alpha), b_\alpha, k_\alpha)_{\alpha<\kappa}$.
	\end{enumerate}
\end{fact}

\begin{fact} \label{sequences}
	\cite[Lemma 7.9]{Cher14}
	Let $(c_i)_{i\in I}$ be an indiscernible sequence of singletons. Then the function $(i,j) \mapsto v(c_j-c_i)$ with $i<j$ satisfies one of the following:
	\begin{enumerate}
		\item It is strictly increasing depending only on $i$ (so $(c_i)_{i\in I}$ is pseudo-convergent),
		\item It is strictly decreasing depending only on $j$ (so $(c_i)_{i\in I}$ taken in the reverse direction is pseudo-convergent), or
		\item It is constant (in this case $(c_i)_{i\in I}$ is referred to as a ``fan'').
	\end{enumerate}
\end{fact}

Because there are definable surjections $v:K^\times\twoheadrightarrow vK$ and $\ac:K\twoheadrightarrow Kv$, we only need to consider inp-patterns where the variable is in the VF-sort. Combining Proposition \ref{betterqe} and Fact \ref{boolean}, we can already focus only on inp-patterns with very tame formulas, but before we can prove the main result, we need to restrict our focus to even more tame inp-patterns.

Throughout this section, we will write $(\phi_\alpha(x, y_\alpha, z_\alpha), (b_\alpha, c_\alpha))_{\alpha<\kappa}$ for indiscernible inp-patterns, where for each $\alpha<\kappa$
\begin{enumerate}
	\item $x$ is a singleton in the VF-sort,
	\item $y_\alpha$ is a tuple of VG-sort and RF-sort variables (we will use $y_\alpha^\VG$ to indicate the subtuple of $y_\alpha$ containing precisely the VG-sort parameters of $y_\alpha$, in the same order; similarly for $y_\alpha^\RF$),
	\item $b_\alpha = (b_{\alpha,i})_{i<\omega}$ is a sequence of VG-sort and RF-sort parameters corresponding to $y_\alpha$ (we will use $b_{\alpha,i}^\VG$ and  $b_{\alpha,i}^\RF$ to indicate the subtuples of $b_{\alpha,i}$ corresponding to $y_\alpha^\VG$ and $y_\alpha^\RF$, respectively),
	\item $z_\alpha$ is a tuple of VF-sort variables, and
	\item $c_\alpha = (c_{\alpha,i})_{i<\omega}$ is a sequence of VF-sort parameters corresponding to $z_\alpha$.
\end{enumerate}

We begin with a technical lemma which will allow us to replace an inp-pattern with another of the same depth with certain VF-sort parameters removed.

\begin{lemma} \label{technical}
	Assume $T$ and $\K$ are as above and let $(\psi_\alpha(x,y_\alpha, z_\alpha, z'_\alpha), (b_\alpha, c_\alpha, c'_\alpha))_{\alpha<\kappa}$ be an indiscernible inp-pattern with $x$ a singleton in the valued field sort. Assume moreover that for each $\alpha<\kappa$ there exist finitely many terms $\{t_\alpha^j : 1\leq j<n_\alpha\}$ in the $\VG$-sort and $\RF$-sort such that
	\begin{enumerate}
		\item none of the terms $t_\alpha^j$ contain the variable $x$, and
		\item viewing $\psi_\alpha$ as a string of symbols, whenever a variable from the tuple $z_\alpha$ occurs in $\psi_\alpha$, that occurrence is contained in a substring of $\psi_\alpha$ equal to one of the terms $t_\alpha^j$.
	\end{enumerate}
	Then for each $\alpha<\kappa$ there exists a tuple of $\VG$-sort and $\RF$-sort variables $y'_\alpha$, a corresponding parameter sequence $b'_\alpha$, and a formula $\phi'_\alpha(x,y'_\alpha,z'_\alpha)$ such that $((\phi'_\alpha(x,y'_\alpha,z'_\alpha), (b'_\alpha,c'_\alpha))_{\alpha<\kappa}$ is an indiscernible inp-pattern of the same depth $\kappa$.
\end{lemma}

\begin{proof}	
	We build $\phi'_\alpha$ from $\psi_\alpha$ by introducing new variable symbols to replace the terms containing $z_\alpha$. Fix $\alpha<\kappa$ and let $w_\alpha^1,\ldots,w_\alpha^{n_\alpha}$ be new variable symbols. For each $j\leq n_\alpha$ and $i<\omega$, let $d_{\alpha,i}^j = t_\alpha^j(b_{\alpha,i},c_{\alpha,i},c'_{\alpha,i})$. Let $y'_\alpha = y_\alpha w_\alpha^1\ldots w_\alpha^{n_\alpha}$, and let $b'_{\alpha,i} = b_{\alpha,i}d_{\alpha,i}^1\ldots d_{\alpha,i}^{n_\alpha}$ such that the variables $w_\alpha^j$ correspond to the parameters $d_{\alpha,i}^j$.
	
	Let $\phi'_\alpha(x,y'_\alpha,z'_\alpha)$ be the same formula as $\psi(x,y_\alpha,z_\alpha,z'_\alpha)$, but with the new variables $w_\alpha^1,\ldots,w_\alpha^{n_\alpha}$ replacing each occurrence of the terms $t_\alpha^1,\ldots,t_\alpha^{n_\alpha}$. By assumption, $z_\alpha$ no longer occurs in the formula after making this substitution, and so we can remove it from the list of variables.
	
	Apply the above process for each $\alpha<\kappa$, and note that
	\[ \K\models \phi'_\alpha(x,b'_{\alpha,i},c'_{\alpha,i}) \leftrightarrow \psi_\alpha(x,b_{\alpha,i},c_{\alpha,i},c'_{\alpha,i}) \]
	for each $\alpha<\kappa$ and $i<\omega$. Since each coordinate of $b'_{\alpha,i}$ was built from a term including only parameters from $(b_{\alpha,i}, c_{\alpha,i}, c'_{\alpha,i})$, the array $\{(b'_\alpha,c'_\alpha) : \alpha<\kappa\}$ is mutually indiscernible. Then applying Fact \ref{boolean}(2),  $(\phi'_\alpha(x,y'_\alpha,z'_\alpha), (b'_\alpha,c'_\alpha))_{\alpha<\kappa}$ is an indiscernible inp-pattern of depth $\kappa$, as desired.
\end{proof}

We will use Lemma \ref{technical} to prove a sequence of propositions towards a proof of Theorem \ref{mainthm}. The propositions will allow us to replace a general inp-pattern for $T$ with progressively less complicated inp-patterns, until the pattern is sufficiently simple for us to deduce the theorem.

\begin{prop} \label{singleton}
	Assume $T$ and $\K$ are as above and let $(\phi_\alpha(x,y_\alpha, z_\alpha), (b_\alpha, c_\alpha))_{\alpha<\kappa}$ be an indiscernible inp-pattern with $x$ a singleton in the valued field sort. Then we can construct a new inp-pattern $(\phi'_\alpha(x,y'_\alpha, z'_\alpha), (b'_\alpha, c'_\alpha))_{\alpha<\kappa}$ of the same depth, such that each formula $\phi'_\alpha$ has the form
	\[ \chi_\alpha\big(v(x-z'_\alpha), (y'_\alpha)^\VG\big) \wedge \rho_\alpha\big(\ac(x-z'_\alpha), (y'_\alpha)^\RF\big) \]
	and such that:
	\begin{enumerate}
		\item $\chi_\alpha$ and $\rho_\alpha$ are formulas in $\Lvg$ and $\Lrf$, respectively, and
		\item $z'_\alpha$ is a singleton in the valued field sort.
	\end{enumerate}
\end{prop}

\begin{proof}
	We wish to apply Lemma \ref{technical}. Fix some $\alpha<\kappa$. By Proposition \ref{betterqe}, Fact \ref{boolean}, and the indiscernibility of $(b_\alpha, c_\alpha)$, we may assume that $\phi_\alpha(x,y_\alpha,z_\alpha)$ has the form
	\[ \chi_\alpha\big(v(x-z_\alpha^1),\ldots,v(x-z_\alpha^n), y_\alpha^\VG\big) \wedge \rho_\alpha\big(\ac(x-z_\alpha^1),\ldots, \ac(x-z_\alpha^n), y_\alpha^\RF\big) \]
	where $z_\alpha = (z_\alpha^1, \ldots, z_\alpha^n)$ for some $n\in\mathbb N$.
	
	If $n=1$ then we may take $\psi_\alpha(x,y_\alpha,z_\alpha,z'_\alpha) = \phi_\alpha(x,y_\alpha,z'_\alpha)$ ($z_\alpha$ will be an unused variable), $c'_\alpha = c_\alpha$, and the set of terms $t_\alpha^j$ to be the empty set. Otherwise, let $\theta = \phi_\alpha$ and fix a realization $a$ of $\{\phi_\alpha(x,b_{\alpha,0}, c_{\alpha,0}) : \alpha<\kappa\}$.
	
	Let $I_\theta$ be the set of indices $i$ such that either $v(x-z_\alpha^i)$ or $\ac(x-z_\alpha^i)$ appears in $\theta$. We proceed recursively, at each step replacing $\theta$ with a new formula $\theta'$ such that $|I_{\theta'}| = |I_\theta|-1$. Set $r = \min(I_\theta)$ and $s = \max(I_\theta)$; we break into cases based on the relationships between $a$, $c_{\alpha,0}^r$, and $c_{\alpha,0}^s$, following the proof of Lemma 7.12 of \cite{Cher14}.
	
	\newcase{1} If $v(a-c_{\alpha,0}^r) < v(c_{\alpha,0}^s- c_{\alpha,0}^r)$, then $v(a-c_{\alpha,0}^r) = v(a-c_{\alpha,0}^s)$ and $\ac(a-c_{\alpha,0}^r) = \ac(a-c_{\alpha,0}^s)$ by Fact \ref{12-easy}. Take $\theta'(x,y_\alpha,z_\alpha)$ to be the conjunction of
	\begin{enumerate}
		\item $\theta$ with each occurrence of $v(x-z_\alpha^s)$ replaced by $v(x-z_\alpha^r)$ and each occurrence of $\ac(x-z_\alpha^s)$ replaced by $\ac(x-z_\alpha^r)$, and
		\item The formula $v(x-z_\alpha^r) < v(z_\alpha^s-z_\alpha^r)$.
	\end{enumerate}
	
	\nogapcase{2} If $v(a-c_{\alpha,0}^r) > v(c_{\alpha,0}^s - c_{\alpha,0}^r)$ then $v(a-c_{\alpha,0}^s) = v(c_{\alpha,0}^s-c_{\alpha,0}^r)$ and $\ac(a-c_{\alpha,0}^s) = \ac(c_{\alpha,0}^s-c_{\alpha,0}^r)$. Take $\theta'(x,y_\alpha,z_\alpha)$ to be the conjunction of
	\begin{enumerate}
		\item $\theta$ with each occurrence of $v(x-z_\alpha^s)$ replaced by $v(z_\alpha^s-z_\alpha^r)$ and each occurrence of $\ac(x-z_\alpha^s)$ replaced by $\ac(z_\alpha^s-z_\alpha^r)$, and
		\item The formula $v(x-z_\alpha^r) > v(z_\alpha^s-z_\alpha^r)$.
	\end{enumerate}
	
	\nogapcase{3} If $v(a-c_{\alpha,0}^s) < v(c_{\alpha,0}^s- c_{\alpha,0}^r)$, proceed symmetrically to case 1.
	
	\newcase{4} If $v(a-c_{\alpha_0}^s) > v(c_{\alpha,0}^s - c_{\alpha,0}^r)$, proceed symmetrically to case 2.
	
	\newcase{5} If $v(a-c_{\alpha,0}^r) = v(a-c_{\alpha,0}^s) = v(c_{\alpha,0}^s - c_{\alpha,0}^r)$ then by Fact \ref{12-easy} again, we must have $\ac(a-c_{\alpha,0}^s) = \ac(a-c_{\alpha,0}^r) - \ac(c_{\alpha,0}^s-c_{\alpha,0}^r)$. Take $\theta'(x,y_\alpha,z_\alpha)$ to be the conjunction of
	\begin{enumerate}
		\item $\theta$ with each occurrence of $v(x-z_\alpha^s)$ replaced by $v(z_\alpha^s-z_\alpha^r)$ and each occurrence of $\ac(x-z_\alpha^s)$ replaced by $\ac(x-z_\alpha^r) - \ac(z_\alpha^s-z_\alpha^r)$, and
		\item The formula $v(x-z_\alpha^r) = v(z_\alpha^s-z_\alpha^r) \wedge \ac(x-z_\alpha^r) \neq \ac(z_\alpha^s-z_\alpha^r)$.
	\end{enumerate}
	
	\noindent Note that in each case, we have $\K\models \theta'(a,b_{\alpha,0},c_{\alpha,0})$ by construction, and that $|I_{\theta'}| = |I_\theta|-1$. If $|I_{\theta'}| = 1$, let $r$ be the single index in $I_{\theta'}$, set $c'_\alpha = (c_{\alpha,i}^r)_{i<\omega}$, and set $\psi_\alpha(x,y_\alpha,z_\alpha,z_\alpha')$ to be $\theta'$ with each occurrence of $v(x-z_\alpha^r)$ replaced by $v(x-z'_\alpha)$ and each occurrence of $\ac(x-z_\alpha^r)$ replaced by $\ac(x-z'_\alpha)$. Otherwise, repeat the process recursively with $\theta'$ in place of $\theta$.
	
	Since $c_\alpha'$ is a subtuple of $c_\alpha$, the array $\{(b_\alpha, c_\alpha,c'_\alpha) : \alpha<\kappa\}$ is mutually indiscernible. By choice of $\theta'$ and $\psi_\alpha$, any realization of $\Psi_\alpha = \{\psi_\alpha(x,b_{\alpha,i}, c_{\alpha,i}, c'_{\alpha,i}) : i<\omega\}$ would also be a realization of $\{\phi_\alpha(x,b_{\alpha,i},c_{\alpha,i}) : i<\omega\}$, and so $\Psi_\alpha$ is inconsistent. Thus, $(\psi_\alpha(x,y_\alpha,z_\alpha,z'_\alpha), (b_\alpha,c_\alpha,c'_\alpha))_{\alpha<\kappa}$ is an indiscernible inp-pattern.
	
	If we then take the collection $\{v(z_\alpha^i-z_\alpha^j) : 1\leq i,j\leq n\} \cup \{\ac(z_\alpha^i-z_\alpha^j) : 1\leq i,j\leq n\}$ for the set of terms $t^i_\alpha$, we can apply Lemma \ref{technical} to obtain $((\phi'_\alpha(x,y'_\alpha,z'_\alpha), (b'_\alpha,c'_\alpha))_{\alpha<\kappa}$, a new indiscernible inp-pattern of depth $\kappa$. By choice of $\psi_\alpha$ and the fact that $z'_\alpha$ is a singleton for all $\alpha<\kappa$, the formulas in the new inp-pattern have the desired form.
\end{proof}

We have just shown that we can replace any inp-pattern with one in which there is only one $\VF$-sort parameter in each row. In the next two propositions, we show that we can find a new inp-pattern in which the $\VF$-sort parameter is constant within each row, and then one in which there is no $\VF$-sort parameter in any row.

\begin{prop}
	Assume $T$ and $\K$ are as above, and let $(\phi_\alpha(x,y_\alpha,z_\alpha), (b_\alpha,c_\alpha))_{\alpha<\kappa}$ be an indiscernible inp-pattern with $x$ a singleton in the valued field sort. Then we can construct a new indiscernible inp-pattern $(\phi'_\alpha(x,y'_\alpha,z'_\alpha), (b'_\alpha,c'_\alpha))_{\alpha<\kappa}$ of the same depth, such that for each $\alpha<\kappa$,
	\begin{enumerate}
		\item the formula $\phi'_\alpha$ has the form described in Proposition \ref{singleton}, and
		\item the $\VF$-sort sequence $c'_\alpha = (c'_{\alpha,i})_{i<\omega}$ is a constant sequence of singletons.
	\end{enumerate}
\end{prop}

\begin{proof}
	First, by applying Proposition \ref{singleton}, we may assume that each $\phi_\alpha$ has the form described in that proposition. We again wish to apply Lemma \ref{technical}. From the conclusion of Proposition \ref{singleton}, each $c_\alpha$ is an indiscernible sequence of singletons.
	
	For every $\alpha<\kappa$, fix an element $c_{\alpha,\infty}$ such that
	\begin{itemize}
		\item if $c_\alpha$ is pseudo-convergent or a fan, then $c_\alpha^+ = (c_{\alpha,0}, c_{\alpha,1}, c_{\alpha,2},\ldots,c_{\alpha,\infty})$ is indiscernible, and
		\item if $c_\alpha$ taken in the reverse order is pseudo-convergent, then $c_\alpha^+ = (c_{\alpha,\infty}, c_{\alpha,0}, c_{\alpha,1},\ldots)$ is indiscernible.
	\end{itemize}
	By compactness, we may assume that the set of sequences $\{(b_\alpha,c_\alpha^+) : \alpha<\kappa\}$ is mutually indiscernible. Take $c'_\alpha = (c'_{\alpha,i})_{i<\omega}$ to be the constant sequence $c'_{\alpha,i} = c_{\alpha,\infty}$ for all $i<\omega$ and all $\alpha<\kappa$.
	
	Let $z'_\alpha$ be a new variable symbol corresponding to $c'_\alpha$ and fix a realization $a$ of $\{\phi_\alpha(x,b_{\alpha,0},c_{\alpha,0}) : \alpha<\kappa\}$. To find the formulas $\psi_\alpha(x,y_\alpha,z_\alpha,z'_\alpha)$ needed for Lemma \ref{technical}, we split into cases based on the relationship between $v(a-c_{\alpha,\infty})$ and $v(c_{\alpha,0}-c_{\alpha,\infty})$.
	
	Fix $\alpha<\kappa$. For legibility, we will write $b_i,c_i,c_\infty$ in place of $b_{\alpha,i},c_{\alpha,i},c_{\alpha,\infty}$ in the cases below. We will clearly have $\K\models \psi_\alpha(a,b_0,c_0,c_\infty)$ by choice of $\psi_\alpha$ in each case. Once $\psi_\alpha$ is chosen, set $\Psi_\alpha(x) = \{\psi_\alpha(x,b_i,c_i,c_\infty) : i<\omega\}$.
	
	\newcase{1} If $v(a-c_\infty) < v(c_0-c_\infty)$ then $v(a-c_0) = v(a-c_\infty)$ and $\ac(a-c_0) = \ac(a-c_\infty)$ by Fact \ref{12-easy}. Let $\psi_\alpha(x,y_\alpha,z_\alpha,z'_\alpha)$ be the formula
	\[ v(x-z'_\alpha) < v(z_\alpha-z'_\alpha) \wedge \chi\big(v(x-z'_\alpha), y_\alpha^\VG\big) \wedge \rho\big(\ac(x-z'_\alpha), y_\alpha^\RF\big) .\]
	Note that any realization of $\Psi_\alpha(x)$ would also be a realization of $\{\phi_\alpha(x,b_i,c_i) : i<\omega\}$, so $\Psi_\alpha(x)$ is inconsistent.
	
	\newcase{2} If $v(a-c_\infty) > v(c_0-c_\infty)$ then by Fact \ref{12-easy}, $v(a-c_0) = v(c_\infty-c_0)$ and $\ac(a-c_0) = \ac(c_\infty - c_0)$, so $K\models \phi(c_\infty, b_0, c_0)$. Then by indiscernibility, $c_\infty$ realizes $\{\phi_\alpha(x,b_i,c_i) : i<\omega\}$, contradicting the inconsistency of that row of the inp-pattern. Thus, case 2 cannot occur.
	
	\newcase{3} Assume $v(a-c_\infty) = v(c_0-c_\infty)$. In this case, we need to split into subcases based on the form of the sequence $(c_i)_{i<\omega}$ and the relationship between $\ac(a-c_\infty)$ and $\ac(c_0-c_\infty)$.
	
	\newcase{3a} If $\ac(a-c_\infty) \neq \ac(c_0-c_\infty)$ then $v(a-c_0) = v(a-c_\infty) = v(c_0-c_\infty)$, so $\ac(a-c_0) = \ac(a-c_\infty) - \ac(c_0-c_\infty)$. Let $\psi_\alpha(x,y_\alpha,z_\alpha,z'_\alpha)$ be the formula
	\[ v(x-z'_\alpha) = v(z_\alpha-z'_\alpha) \wedge \ac(x-z'_\alpha) \neq \ac(z_\alpha-z'_\alpha) \]
	\[ \wedge \chi\big(v(x-z'_\alpha), y_\alpha^\VG\big) \wedge \rho\big(\ac(x-z'_\alpha)-\ac(z_\alpha-z'_\alpha), y_\alpha^\RF\big) .\]
	As in Case 1, note that any realization of $\Psi_\alpha(x)$ would also be a realization of $\{\phi_\alpha(x,b_i,c_i) : i<\omega\}$, so $\Psi_\alpha(x)$ is inconsistent.
	
	\newcase{3b} Suppose $(c_i)_{i<\omega}$ or its reversal is pseudo-convergent and let $\psi_\alpha(x,y_\alpha,z_\alpha,z'_\alpha)$ be the formula $v(x-z'_\alpha) = v(z_\alpha-z'_\alpha)$. It is easy to check that $c_\infty$ is a pseudo-limit of $(c_i)_{i<\omega}$ or its reversal, whichever is pseudo-convergent, and so $v(c_i-c_\infty) \neq v(c_j-c_\infty)$ whenever $i \neq j$. Thus, for any $d\in K$, it is impossible for $v(d-c_\infty)$ to be equal to both $v(c_i-c_\infty)$ and $v(c_j-c_\infty)$; in other words, $\Psi_\alpha(x)$ is inconsistent.
	
	\newcase{3c} Finally, by Fact \ref{sequences}, suppose $(c_i)_{i<\omega}$ is a fan and $\ac(a-c_\infty) = \ac(c_0-c_\infty)$. Let $\psi_\alpha(x,y_\alpha,z_\alpha,z'_\alpha)$ be the formula
	\[ v(x-z'_\alpha) = v(z_\alpha-z'_\alpha) \wedge \ac(x-z'_\alpha) = \ac(z_\alpha-z'_\alpha) .\]
	Since $c_\infty$ will be an element of the fan, $\ac(c_i-c_\infty) \neq \ac(c_j-c_\infty)$ for any $i\neq j$. Thus, for any $d\in K$, it is impossible for $\ac(d-c_\infty)$ to be equal to both $\ac(c_i-c_\infty)$ and $\ac(c_j-c_\infty)$, which means $\Psi_\alpha(x)$ is inconsistent.
	
	\bigskip\noindent As noted above, Case 2 cannot occur. In each other case, we have chosen $\psi_\alpha(x,y_\alpha,z_\alpha,z'_\alpha)$ so that $\Psi_\alpha(x)$ is inconsistent and $\K \models \psi_\alpha(a,b_{\alpha,0},c_{\alpha,0},c'_{\alpha,0})$. In addition, by choice of $c'_\alpha$, the array $\{(b_\alpha,c_\alpha,c'_\alpha) : \alpha<\kappa\}$ is mutually indiscernible. Thus, $\{\psi_\alpha(x,y_\alpha,z_\alpha,z'_\alpha), (b_\alpha,c_\alpha,c'_\alpha)\}_{\alpha<\kappa}$ is an indiscernible inp-pattern.
	
	Finally, the terms $t_\alpha^1 = v(z_\alpha-z'_\alpha)$ and $t_\alpha^2 = \ac(z_\alpha-z'_\alpha)$ satisfy the remaining conditions of Lemma \ref{technical}, and we obtain a new inp-pattern $((\phi'_\alpha(x,y_\alpha',z_\alpha'), (b'_\alpha,c'_\alpha))_{\alpha<\kappa}$ in which the $\VF$-sort parameter sequence of each row of the new inp-pattern is $c'_\alpha$, a constant sequence of singletons. Moreover, each $\psi_\alpha$ has the form described in Proposition \ref{singleton} by construction, and $\phi'_\alpha$ inherits this form since it is obtained from $\psi_\alpha$ through a substitution of terms. Thus, the new inp-pattern has the desired form.
\end{proof}


\begin{prop} \label{noshift}
	Assume $T$ and $\K$ are as above, and let $(\phi_\alpha(x,y_\alpha,z_\alpha), (b_\alpha,c_\alpha))_{\alpha<\kappa}$ be an indiscernible inp-pattern with $x$ a singleton in the valued field sort. Then we can construct a new indiscernible inp-pattern $(\phi'_\alpha(x,y'_\alpha), (b'_\alpha))_{\alpha<\kappa}$ of the same depth, such that for each $\alpha<\kappa$, the formula $\phi'_\alpha$ has the form
	\[ \chi_\alpha\big(v(x), (y'_\alpha)^\VG\big) \wedge \rho_\alpha\big(\ac(x), (y'_\alpha)^\RF\big) ,\]
	where $\chi_\alpha$ and $\rho_\alpha$ are formulas in $\Lvg$ and $\Lrf$, respectively.
\end{prop}

\begin{proof}
	From the previous propositions, we may assume each $\phi_\alpha(x,y_\alpha,z_\alpha)$ has the form
	\[ \chi_\alpha\big(v(x-z_\alpha), y_\alpha^\VG\big) \wedge \rho_\alpha\big(\ac(x-z_\alpha), y_\alpha^\RF\big) \]
	and that for each $\alpha<\kappa$, $z_\alpha$ is a singleton and $c_\alpha$ is a constant sequence. Throughout this proof, we will identify a constant sequence with its value. We will again apply Lemma \ref{technical}. Let $a$ be some realization of $\{\phi_\alpha(x,b_{\alpha,0},c_\alpha) : \alpha<\kappa\}$.
	
	For any $\alpha,\beta<\kappa$ such that $v(a-c_\alpha) < v(a-c_\beta)$, we have $v(a-c_\alpha) = v(c_\beta-c_\alpha)$ and $\ac(a-c_\alpha) = \ac(c_\beta-c_\alpha)$ by Fact \ref{12-easy}. Then, since $K\models \phi_\alpha(a, b_{\alpha,0}, c_\alpha)$, we have $K\models \phi_\alpha(c_\beta, b_{\alpha,0}, c_\alpha)$. But then by mutual indiscernibility, $K\models \phi_\alpha(c_\beta, b_{\alpha,i}, c_\alpha)$ for all $i<\omega$, contradicting the inconsistency of the row $\alpha$.
	
	Thus, $v(a-c_\alpha)$ is constant for all $\alpha<\kappa$; in particular, it is equal to $v(a-c_0)$. For each $\alpha$, let $c'_\alpha = c_\alpha-c_0$, and let $a' = a-c_0$. Since $(b_{\alpha,i},c_{\alpha,i})_{i<\omega}$ is indiscernible over $c_0$ for all $\alpha<\kappa$ (including $\alpha=0$, since $c_0=c_{0,i}$ for all $i<\omega$), the array obtained by replacing $c_\alpha$ with $c'_\alpha$ is still an inp-pattern, and $a'$ will be a realization of the first column. To simplify notation, assume that $c_0 = 0$, so $a'=a$ and $c'_\alpha = c_\alpha$.
	
	Now $v(a-c_\alpha) = v(a-c_0) = v(a)$ for all $\alpha<\kappa$, and so $\ac(a-c_\alpha)$ equals either $\ac(a)$ or $\ac(a) - \ac(c_\alpha)$, depending on whether $v(a) < v(c_\alpha)$ or $v(a) = v(c_\alpha)$; the case where $v(a)>v(c_\alpha)$ is impossible since $v(a-c_\alpha) = v(a)$. We again split into cases in order to define formulas $\psi_\alpha(x,y_\alpha,z_\alpha,z'_\alpha)$ for $\alpha<\kappa$.
	
	\newcase{1} If $v(a) < v(c_\alpha)$, take $\psi(x,y_\alpha,z_\alpha)$ to be the formula
	\[ v(x) < v(z_\alpha) \wedge \chi(v(x), y_\alpha^\VG) \wedge \rho(\ac(x), y_\alpha^\RF) .\]
	
	\nogapcase{2} If $v(a) = v(c_\alpha)$, take $\psi_\alpha(x,y_\alpha,z_\alpha)$ to be the formula
	\[ v(x) = v(z_\alpha) \wedge \chi(v(x), y_\alpha^\VG) \wedge \rho(\ac(x) - \ac(z_\alpha), y_\alpha^\RF) .\]
	
	\noindent In either case, $\K\models \psi_\alpha(a,b_{\alpha,0},c_\alpha)$ and any realization of $\Psi_\alpha = \{\psi_\alpha(x,b_{\alpha,i}, c_{\alpha,i}) : i<\omega\}$ would also be a realization of $\{\phi_\alpha(x,b_{\alpha,i},c_{\alpha,i}) : i<\omega\}$. Thus, $\Psi_\alpha$ is inconsistent and $\{\psi_\alpha(x,y_\alpha,z_\alpha), (b_\alpha,c_\alpha)\}_{\alpha<\kappa}$ is an indiscernible inp-pattern. Take $v(z_\alpha)$ and $\ac(z_\alpha)$ for the terms $t_\alpha^i$.
	
	Then, setting $z'_\alpha$ and $c'_\alpha$ to be empty tuples, we may apply Lemma \ref{technical} to obtain a new inp-pattern $((\phi'_\alpha(x,y_\alpha'), (b'_\alpha))_{\alpha<\kappa}$ with no VF-sort parameter sequences, and in which each formula has the desired form.
\end{proof}

Now that we can reduce to inp-patterns with no $\VF$-sort parameters, we can prove the main theorem.

\begin{theorem} \label{mainthm}
	Suppose $T$ is a theory of henselian valued fields in $\Lpas$ admitting relative quantifier elimination. Then
	\[ \bdn(T) = \bdn(T_{\VG}) + \bdn(T_{\RF}) ,\]
	where $T_{\VG}$ and $T_{\RF}$ are the induced theories on the value group and residue field, respectively.
\end{theorem}

\begin{proof}
	We begin by showing that $\bdn(T) \leq \bdn(T_\VG) + \bdn(T_\RF)$. Suppose that $(\phi_\alpha(x,y_\alpha), b_\alpha)_{\alpha<\kappa}$ is an indiscernible inp-pattern for $T$. If $x$ is a $\VG$-sort variable then we can obtain a new inp-pattern $(\phi'_\alpha(x',y_\alpha), b_\alpha)_{\alpha<\kappa}$ with $x'$ a $\VF$-sort variable by taking $\phi'_\alpha(x',y_\alpha) = \phi_\alpha(v(x'),y_\alpha)$. A similar substitution with $\ac(x')$ can replace an $\RF$-sort variable with a $\VF$-sort variable.
	
	Thus, we may assume without loss of generality that $x$ is in the valued field sort. By Proposition \ref{noshift}, we may further assume that for each $\alpha<\kappa$, $y_\alpha$ has no $\VF$-sort component and $\phi_\alpha(x,y_\alpha)$ has the form
	\[ \chi_\alpha\big(v(x),y_\alpha^\VG\big) \wedge \rho_\alpha\big(\ac(x), y_\alpha^\RF\big) \]
	where $\chi_\alpha\in \Lvg$ and $\rho_\alpha\in \Lrf$.
	
	Suppose that for some $\alpha<\kappa$, the sets $X_\alpha(x) = \{\chi_\alpha(v(x), b_{\alpha,i}^\VG) : i<\omega\}$ and $P_\alpha(x) = \{\rho_\alpha(\ac(x), b_{\alpha,i}^\RF) : i<\omega\}$ are both consistent, say they are realized by elements $c$ and $d$, respectively. Then by Fact \ref{orthogonal}, there exists an element $a$ with $v(a) = v(c)$ and $\ac(a) = \ac(d)$. But then $a$ would be a realization of $X_\alpha(x) \cup P_\alpha(x)$, and so would also be a realization of $\{\phi_\alpha(x, b_{\alpha,i}) : i<\omega\}$, contradicting the inconsistency of the row.
	
	Thus, we can write $\kappa = G\cup R$, where $\alpha\in G$ if $X_\alpha(x)$ is inconsistent, and $\alpha\in R$ if $P_\alpha(x)$ is inconsistent. Then for new variable symbols $z$ and $w$, $(\chi_\alpha(z,y_\alpha^\VG), b_\alpha^\VG)_{\alpha\in G}$ is an inp-pattern in $vK$ and $(\rho_\alpha(w,y_\alpha^\RF), b_\alpha^\RF)_{\alpha\in R}$ is an inp-pattern in $Kv$, so
	\[ \kappa = G\cup R = |G\cup R| \leq |G| + |R| \leq \bdn(T_{\VG}) + \bdn(T_{\RF}) .\]
	Since $\bdn(T)$ is the supremum of all such $\kappa$, we have $\bdn(T) \leq \bdn(T_{\VG}) + \bdn(T_{\RF})$.
	
	For the reverse inequality, let $(\chi_\alpha(z,y_\alpha), b_\alpha, k_\alpha)_{0 \leq \alpha < \kappa}$ and $(\rho_\alpha(w,y_\alpha), b_\alpha, k_\alpha)_{\kappa \leq \alpha < \lambda}$ be inp-patterns for $T_\VG$ and $T_\RF$; we do not make any assumption of indiscernibility. For each $0\leq\alpha < \kappa$, let $\phi_\alpha(x,y_\alpha)$ be the formula $\chi_\alpha(v(x),y_\alpha)$, and for each $\kappa \leq \alpha<\lambda$, let $\phi_\alpha(x,y_\alpha)$ be the formula $\rho_\alpha(\ac(x), y_\alpha)$. We claim that $(\phi_\alpha(x,y_\alpha), b_\alpha, k_\alpha)_{0\leq \alpha<\lambda}$ is an inp-pattern for $K$.
	
	First, note that each row is $k_\alpha$-inconsistent, since we started with inp-patterns for $T_\VG$ and $T_\RF$. Fix any function $\eta: \lambda\to\omega$. If $\gamma\in vK$ and $c\in Kv$ are realizations of $\{\chi_\alpha(z,b_{\alpha,\eta(\alpha)}) : 0\leq \alpha<\kappa\}$ and $\{\rho_\alpha(w,b_{\alpha,\eta(\alpha)}) : \kappa\leq \alpha<\lambda\}$, respectively, then any element $a\in K$ with $v(a) = \gamma$ and $\ac(a) =c$ will realize $\{\phi_\alpha(x,b_{\alpha,\eta(\alpha)}) : 0\leq \alpha < \lambda\}$.
	
	Thus, $(\phi_\alpha(x,y_\alpha), b_\alpha, k_\alpha)_{0\leq \alpha<\lambda}$ is an inp-pattern for $T$, which means $\lambda \leq \bdn(T)$. Since $\lambda$ is the sum of the depths of arbitrary inp-patterns for $T_\VG$ and $T_\RF$, taking the supremum over all such inp-patterns yields $\bdn(T_\VG) + \bdn(T_\RF) \leq \bdn(T)$, completing the proof.
\end{proof}

As an immediate consequence of the theorem, we get that \ntp transfers from $T_{\VG}$ and $T_{\RF}$ to $T$. This generalizes \cite[Theorem 7.6]{Cher14} from equicharacteristic zero to any characteristic, provided the theory has relative quantifier elimination.

\begin{cor}
	Let $T$ be a theory of henselian valued fields in $\Lpas$ admitting relative quantifier elimination. Then $T$ is \ntp if and only if $T_{\VG}$ and $T_{\RF}$ are. The same is true for the statements ``$T$ is strong'' and ``$T$ has finite burden.''
\end{cor}



Note that equality only holds in the theorem when working in the Denef-Pas language; if $T'$ is the theory of a reduct of a model of $T$ (for example, if $T'$ is the usual one-sorted valued field language $\Ldiv$), then we only have the inequality
\[ \bdn(T') \leq \bdn(T) = \bdn(T_{\VG}) + \bdn(T_{\RF}) .\]

We know that both $\bdn(T') < \bdn(T)$ and $\bdn(T') = \bdn(T)$ are possible, depending on the choice of $T$ and $T'$:
\begin{enumerate}
	\item The burden of ACVF in $\Ldiv$ is 1, but the burden of ACVF in $\Lpas$ is 2.
	\item The burden of $\Th(\mathbb Q_p)$ is 1 in both $\Ldiv$ and $\Lpas$, since the residue field of $\mathbb Q_p$ is finite.
\end{enumerate}
It is not known whether there is a valued field $K$ with infinite residue field where equality holds in a reduct of the Denef-Pas language. This will certainly happen if the angular component map is definable in $\Ldiv$, but may occur in other situations as well.

\bibliographystyle{alpha}
\bibliography{../References}

\end{document}